\documentclass[12pt,a4]{article}
\usepackage{amsmath, graphicx}
\usepackage{amstext,amsfonts,amsbsy,eucal,amssymb, amsthm}
\usepackage{color}

\advance \topmargin by -\headheight
\advance \topmargin by -\headsep
     
\evensidemargin \oddsidemargin
\newtheorem{theorem}{Theorem}[section]
\newtheorem{proposition}[theorem]{Proposition}
\newtheorem{corollary}[theorem]{Corollary}

\theoremstyle{definition}

\def\Z{\operatorname{\mathbb{Z}}}

\def\C{\operatorname{\mathbb{C}}}
\def\P{\operatorname{\mathbb{P}}}

\def\pic{\operatorname{{\rm Pic}}}
\def\H{\operatorname{{\mathcal H}}}
\def\E{\operatorname{{\mathcal E}}}

\usepackage{float}
\restylefloat{table}
\usepackage{graphicx}
\DeclareSymbolFont{extraup}{U}{zavm}{m}{n}
\DeclareMathSymbol{\varheart}{\mathalpha}{extraup}{86}
\DeclareMathSymbol{\vardiamond}{\mathalpha}{extraup}{87}

\DeclareFontFamily{U}{rsf}{}
\DeclareFontShape{U}{rsf}{m}{n}{<5> <6> rsfs5 <7> <8> <9> rsfs7 <10-> rsfs10}{}
\DeclareMathAlphabet\Scr{U}{rsf}{m}{n}
\def\C{{\mathbb{C}}}

\def\P{{\mathbb{P}}}

\def\Z{{\mathbb{Z}}}

\def\X{{\mathcal X}}
\def\E{{\mathcal E}}

\def\H{{\mathcal H}}
\def\pic{{\rm{Pic}}}

\usepackage{graphics}
\usepackage{epsfig}
\usepackage{amsmath}
\usepackage{amssymb}
\usepackage{multirow}
\usepackage[all]{xy}

\begin{document}
\begin{center}
{\large \bf An algebraically stable variety for a four-dimensional dynamical system reduced from the lattice super-KdV equation}
\end{center}
\begin{center}

{\bf Adrian Stefan Carstea${}^\dagger$  and Tomoyuki Takenawa${}^\ddag$}\\
\medskip

${}^\dagger$ Department of Theoretical Physics, Institute of Physics and Nuclear Engineering\\ Reactorului 15, 077125, Magurele, Bucharest, Romania\\
E-mail: carstea@gmail.com\\[5pt]

${}^\ddag$ Faculty of Marine Technology, Tokyo University of Marine Science and Technology\\ 2-1-6 Etchu-jima, Koto-ku, Tokyo, 135-8533, Japan\\
E-mail: takenawa@kaiyodai.ac.jp

\end{center}

\begin{abstract}
In a prior paper the authors obtained a four-dimensional discrete integrable dynamical system by the traveling wave reduction from the lattice super-KdV equation in a case of finitely generated Grassmann algebra. The system is a coupling of a Quispel-Roberts-Thompson map and a linear map but does not satisfy the singularity confinement criterion. It was conjectured that the dynamical degree of this system grows quadratically. In this paper, constructing a rational variety where the system is lifted to an algebraically stable map and using the action of the map on the Picard lattice, we prove this conjecture. We also show that invariants can be found through the same technique.    
\end{abstract}

\section{Introduction}

In a prior paper \cite{superQRT}, applying the traveling wave reduction to the lattice super-KdV equation  \cite{fane, liu} in a case of finitely generated Grassmann algebra, the authors obtained a four-dimensional discrete integrable dynamical system 
\begin{align}\label{4d}
\varphi:&\left\{\begin{array}{rcl}
\overline{x_0}&=&x_2\\
\overline{x_1}&=&x_3\\
\overline{x_2}&=&-x_2-x_0+\dfrac{h x_2}{1-x_2}\\
\overline{x_3}&=&-x_1-x_3 + \dfrac{2-x_2+hx_3}{(1-x_2)^2}
\end{array}\right..
\end{align}
This system is a Quispel-Roberts-Thompson (QRT) map, a two dimensional map generating an automorphism of a rational elliptic surface \cite{qrt},  for variables $x_0$, $x_2$ coupled with linear equations for variables $x_1$, $x_3$ with coefficients depending on $x_2$.
This system has two invariants
\begin{align}
I_1=&-hx_0^2 - hx_0x_2 + h^2 x_0 x_2 + h x_0^2 x_2 - h x_2^2 + h x_0 x_2^2\\
\nonumber
I_2=&2 h x_0 + x_0^2 - 2 h x_0 x_1 + 2 h x_2 + x_0 x_2 - h x_1 x_2 + h^2 x_1 x_2
+2 h x_0 x_1 x_2\\& + x_2^2
 + h x_1 x_2^2 - h x_0 x_3 + h^2 x_0 x_3 + h x_0^2 x_3 - 
  2 h x_2 x_3 + 2 h x_0 x_2 x_3,
\end{align}
but does not satisfy the {\it singularity confinement criterion} proposed by Grammaticos-Ramani and their collaborators \cite{sc, rgh}. The examples of this criterion is given in the next section.

In the same paper it is observed that the dynamical degrees of \eqref{4d} grows quadratically. This phenomena is rather unusual, since as reported in \cite{Lafortune2001, Gubbiotti2018}, the dynamical degree grows in the fourth order for generic coupled systems in the form
\begin{align*}
&\left\{
\begin{array}{rcl}
\overline{x_0}&=&f_0(x_0,x_1)\\
\overline{x_1}&=&f_1(x_0,x_1)\\
\overline{x_2}&=&f_2(x_0,x_1,x_2)\\
\overline{x_3}&=&f_3(x_0,x_1,x_2,x_3)
\end{array}
\right.,
\end{align*}
where the system is a QRT map for variables $x_0$ and $x_1$, and $\overline{x_2}$ (resp. $\overline{x_3}$) depends on $x_2$ (resp. $x_3$) linearly  with coefficients depending on ``$x_0$ and $x_1$''  (resp. ``$x_0$, $x_1$ and $x_2$'').
This type of systems is also constructed by generalising the QRT maps and referred to as ``triangular'' in \cite{FK2006}.

In this paper, constructing a rational variety where System \eqref{4d} is lifted to an {\it algebraically stable} map and using the action of the map on the Picard lattice, we prove the above conjecture.
We also show that one can find invariants also using the action on the Picard group.    

In the two-dimensional case, it is known that an autonomous dynamical system defined by a birational map on a projective rational variety (or more generally K\"ahler manifold) can be lifted to either an automorphism or an algebraically stable map on a rational variety by successive blow-ups \cite{DF2001}.
Here, a birational map $\varphi$ from an $N$-dimensional rational variety $\X$ to itself is said to be {\it algebraically stable} if $(\varphi^*)^n({\mathcal D})=(\varphi^n)^*({\mathcal D})$ holds for any divisor class ${\mathcal D}$ on $\X$ and an arbitrary positive integer $n$ \cite{BK2008}. 
These notions are closely related to the notion of singularity confinement criterion.
While a dynamical system that can be lifted to automorphisms satisfies singularity confinement criterion (i.e. all the singularities are confined), a dynamical system that can be lifted only to algebraically stable map does not satisfies the criterion (i.e. there exists a singularity that is not confined).

In studies of higher dimensional dynamical systems, the role of automorphisms is replaced by pseudo-automorphisms, i.e. automorphisms except finite number of subvarieties of codimension at least two \cite{DO1988}.  
In the last decade a few authors studied how to construct algebraic varieties on the level of pseudo-automorphisms \cite{BK2008, TT2009, CT2019}. 
However, since System \eqref{4d} does not satisfy the singularity confinement criterion, it is not expected that it could be lifted to a pseudo-automorphism. To authors knowledge there are no studies (except section Section 7 of \cite{BK2008}, which studies a kind of generalisation of standard Cremona transformation) on construction of an algebraic variety, in which the original system is lifted not to a pseudo-automorphism, but rather to an algebraically stable map using blow-ups along sub-varieties of positive dimensions.
Since the varieties obtained by blow-ups possibly infinitely near depend on the order of blow-ups, this is not a straightforward but a challenging problem.

Since $I_2$ is degree $(1,1)$ for $x_1$, $x_3$, we can restrict the phase space into 3-dimensional one as
\begin{align}\label{3d_system}
\psi:&\left\{
\begin{array}{rcl}
x_0&=&x_2\\ 
x_1&=& \dfrac{I_2 - (x_0^2+x_0x_2+x_2^2)-2h(x_0-x_0x_1+x_2)-hx_1x_2(2x_0+x_2-1+h)}{h(- x_0  +  h x_0 +  x_0^2 - 2  x_2 + 2  x_0 x_2) }\\
x_2&=& -x_2 - x_0 + \dfrac{h x_2}{1-x_2}
\end{array}\right..
\end{align}
We also show that the degree of this 3-dimensional system grows quadratically as well.

\section{Algebraically stable space for the 4D system}

Let us consider System $(\ref{4d})$ on the projective space $(\P^1)^4$. 
In the following, we aim to obtain a four-dimensional rational variety by blowing-up procedure such that the birational map $(\ref{4d})$ is lifted
to an algebraically stable map on the variety.

Let $I(\varphi)$ denote the indeterminacy set of $\varphi$. 
It is known that the mapping  
$\varphi$ is algebraically stable if and only if
there does not exist a positive integer $k$ and a divisor $D$ on ${\mathcal X}$ such that 
\begin{align}
\varphi(D\setminus I(\varphi))\subset I(\varphi^k), \label{cond_AS}
\end{align}
i.e. the image of the generic part of a divisor by $\varphi$ is included in the indeterminate set of $\varphi^k$ (\cite{BK2008, Bayraktar2012}, Proposition~2.3 of \cite{CT2019}).    
See Section 2 of \cite{CT2019} for notations and related theories used here.

The notion of singularity series of dynamics studied by Grammaticos-Ramani and their collaborators is closely related to our procedure. Let us start with a hyper-plane $x_2 = 1+ \varepsilon$, where 
$\varepsilon $ is a small parameter for considering Laurent series expression, and apply $\varphi$, then  we have a ``confined'' sequence of Laurent series:
\begin{align}
\cdots \to &(x_0^{(0)},x_1^{(0)},1+\varepsilon,x_3^{(0)})
\to 
(1,x_3^{(0)}, -h \varepsilon^{-1}, (1+h x_3^{(0)})\varepsilon^{-2})\nonumber\\
\to&
(-h \varepsilon^{-1}, (1+h x_3^{(0)})\varepsilon^{-2}, h \varepsilon^{-1}, -(1+h x_3^{(0)})\varepsilon^{-2})\nonumber\\
\to&
(h \varepsilon^{-1}, -(1+h x_3^{(0)})\varepsilon^{-2}, 1,x_4^{(3)} )
\to (1,x_1^{(4)},x_0^{(0)},x_3^{(4)})\to \cdots,\label{seq1}
\end{align}
where $x_i^{(k)}$'s are complex constants and only the principal term is written for each entry and  
a hyper-surface $x_2=0$ is contracted to lower-dimensional varieties and returned to a hyper-surface $x_0=1$ after 4 steps. 
We can also find a cyclic sequence: 
\begin{align}
&(x_0^{(0)},x_1^{(0)},\varepsilon^{-1},x_3^{(0)})
\to 
(\varepsilon^{-1},x_3^{(0)}, -\varepsilon^{-1}, -x_1^{(0)}-x_3^{(0)} )\nonumber\\
\to&
(\varepsilon^{-1},-x_1^{(0)}-x_3^{(0)} ,x_0^{(0)},x_1^{(0)}) \to
 (x_0^{(0)},x_1^{(0)},\varepsilon^{-1},x_3^{(3)}) \mbox{: returned}, \label{seq2}
\end{align}
where a hyper-surface $x_2=\infty$ is contracted to lower-dimensional varieties and returned to the original  hyper-surface after 3 steps, 
and an ``anti-confined'' sequence: 
\begin{align}
\cdots \to&\left( \left(-1+ \frac{h}{(x_0^{(0)}-1)^2}\right) \varepsilon^{-1},x_1^{(-1)},x_2^{(-1)},\varepsilon^{-1}\right)\nonumber\\
\to& 
(x_0^{(0)},\varepsilon^{-1},x_2^{(0)},x_3^{(0)})
\to 
(x_2^{(0)},x_3^{(0)},x_2^{(1)},\varepsilon^{-1})\nonumber\\
\to &
\left(x_2^{(1)},\varepsilon^{-1}, x_2^{(2)} , \left(-1+\frac{h}{(x_2^{(0)}-1)^{2}}\right)\varepsilon^{-1}\right)
\to \cdots, \label{seq3}
\end{align}
where a lower dimensional variety is blown-up to a hyper-surfaces $x_1=\infty$ and contracted to a lower dimensional variety after 3 steps. 

In the following, in order to avoid anti-confined patterns, we consider ${\mathbb P}^2\times {\mathbb P}^2$ instead of $({\mathbb P}^1)^4$. Although there is a possibility that the anti-confined pattern can be resoluted by some blowing-down procedure, it is not easy to find the actual procedure on the level of coordinates.  

The coordinate system of  ${\mathbb P}^2\times {\mathbb P}^2$ is $(x_0:x_1:1,\ x_2:x_3:1)$, and thus the local coordinate systems essentially consist of $3\times 3=9$ charts: 
\begin{align*}&(x_0, x_1, x_2, x_3), (y_0, y_1, x_2, x_3), (z_0, z_1, x_2, x_3),\nonumber\\
&(x_0, x_1, y_2, y_3), (y_0, y_1, y_2, y_3), (z_0, z_1, y_2, y_3),\nonumber\\
&(x_0, x_1, z_2, z_3), (y_0, y_1, z_2, z_3), (z_0, z_1, z_2, z_3),
\end{align*}
where $y_i$'s and $z_i$'s are 
\begin{align*}
(y_i,y_{i+1})=(x_i^{-1},x_i^{-1}x_{i+1}) \mbox{ and }
(z_i,z_{i+1})=(x_i x_{i+1}^{-1},x_{i+1}^{-1})
\end{align*}
for $i=0,2$.
Then, both the cyclic sequence \eqref{seq2} and the anti-confined sequence  \eqref{seq3} starting with $x_i^{(0)}=\varepsilon^{-1}$ do not appear, but another cyclic sequence
\begin{align}
&(x_0^{(0)},x_1^{(0)},\varepsilon^{-1}, c^{(0)}\varepsilon^{-1})
\to 
(\varepsilon^{-1}, c^{(0)}\varepsilon^{-1}, -\varepsilon^{-1}, -c^{(0)}\varepsilon^{-1})\nonumber\\
\to&
(-\varepsilon^{-1}, -c^{(0)}\varepsilon^{-1}, x_0^{(0)},x_1^{(0)}) \to (x_0^{(0)},x_1^{(0)},\varepsilon^{-1}, c^{(0)}\varepsilon^{-1}) \mbox{: returned}\label{seq4}
\end{align}
appears, where $c^{(0)}$ is also a complex constant.

In order to resolute the singularity appeared in Sequences \eqref{seq1} and \eqref{seq4}, we blow up the rational variety along the sub-varieties to which some divisor is contracted to.  
For Sequences \eqref{seq1}, we have three such sub-varieties whose parametric expressions are
\begin{align*}
V_1: &(x_0,x_1,z_2,z_3)=(P, 1,0,0)\\
V_2: &(z_0,z_1,z_2,z_3)=(0, 0,0,0)\\
V_3: &(z_0,z_1,x_2,x_3)=(0,0,P, 1), 
\end{align*}
where $P$ is a $\C$-valued parameter (independent to another sub-variety's),
while for Sequences \eqref{seq4} we have a sub-variety
\begin{align*}
V_4:& (z_0,z_1,z_2,z_3)=(P, 0,P,0).
\end{align*}
That is, the subvariety $V_1$ is the Zariski closure of $\{(x_0,x_1,x_2,x_3)=(P, 1,0,0)~|~ P\in\C \}$
and $V_4$ is that of of $\{(x_0,x_1,x_2,x_3)=(P, 0,P,0)~|~ P\in\C \}$ and so forth.

Since $V_4$ includes $V_2$, we have the option of blowing-up order. 
In the two dimensional case, resolution is unique and the order is not a matter. But in the higher dimensional case, it affects sensitively to the resulting varieties. Since we only care on the level of  codimension one, the order of blow-ups does not affect the algebraical stability in some cases.   
However, the following results were obtained not in a straightforward manner but by trial and error. 
       
We can resolute the singularity around $V_1$ by the following five blowups:
\begin{align*}
C_1:&(x_0,x_1,z_2,z_3)=(1,P,0,0) \\
& \quad \leftarrow (s_1,t_1,u_1,v_1):=(x_0-1,x_1, z_2(x_0-1)^{-1}, z_3(x_0-1)^{-1}),\\
C_2:&(s_1,t_1,u_1,v_1)=(0,P,Q,0) \\
& \quad \leftarrow (s_2,t_2,u_2,v_2):=(s_1,t_1, u_1, v_1s_1^{-1}),\\
C_3:&(s_2,t_2,u_2,v_2)=(0,P,-h(1+h P)^{-1},Q) \\
& \quad \leftarrow (s_3,t_3,u_3,v_3):=(s_2,t_2, (u_2+h(1+h t_2)^{-1}) s_2^{-1}, v_2),\\
C_4:&(s_3,t_3,u_3,v_3)=(0,P,Q,(1+h P)^{-1}) \\
& \quad \leftarrow (s_4,t_4,u_4,v_4):=(s_3,t_3, u_3, (v_3-(1+h t_3)^{-1}) s_3^{-1}),\\
C_5:&(s_4,t_4,u_4,v_4)=(0,P,Q,(1+h P)^{-2}) \\
& \quad \leftarrow (s_5,t_5,u_5,v_5):=(s_4,t_4, u_4, (v_4-(1+h t_4)^{-2}) s_4^{-1}),
\end{align*}
where only one of the coordinate systems is written for each blowup.
Similarly, we can resolute the singularity around $V_3$ by the following five blowups:
\begin{align*}
C_6:&(z_0,z_1,x_2,x_3)=(0,0,1,P) \\
& \quad \leftarrow (s_6,t_6,u_6,v_6):=(x_2-1,x_3, z_0(x_2-1)^{-1}, z_1(x_2-1)^{-1}),\\
C_7:&(s_6,t_6,u_6,v_6)=(0,P,Q,0) \\
& \quad \leftarrow (s_7,t_7,u_7,v_7):=(s_6,t_6, u_6, v_6s_6^{-1}),\\
C_8:&(s_7,t_7,u_7,v_7)=(0,P,-h(1+h P)^{-1},Q) \\
& \quad \leftarrow (s_8,t_8,u_8,v_8):=(s_7,t_7, (u_7+h(1+h t_7)^{-1}) s_7^{-1}, v_7),\\
C_9:&(s_8,t_8,u_8,v_8)=(0,P,Q,(1+h P)^{-1}) \\
& \quad \leftarrow (s_9,t_9,u_9,v_9):=(s_8,t_8, u_8, (v_8-(1+h t_8)^{-1}) s_8^{-1}),\\
C_{10}:&(s_9,t_9,u_9,v_9)=(0,P,Q,(1+h P)^{-2}) \\
& \quad \leftarrow (s_{10},t_{10},u_{10},v_{10}):=(s_9,t_9, u_9, (v_9-(1+h t_9)^{-2}) s_9^{-1}).
\end{align*} 
We need three blowups for $V_4$:
\begin{align*}
C_{11}:&(z_0,z_1,z_2,z_3)=(0,0,0,0) \\
& \quad \leftarrow (s_{11},t_{11},u_{11},v_{11}):=(z_0,z_1z_0^{-1}, z_2z_0^{-1}, z_3z_0^{-1}),\\
C_{12}:&(s_{11},t_{11},u_{11},v_{11})=(P,0,1,0) \\
& \quad \leftarrow (s_{12},t_{12},u_{12},v_{12}):=(s_{11},t_{11},(u_{11}-1)t_{11}^{-1}, v_{11}t_{11}^{-1}),\\
C_{13}:&(s_{12},t_{12},u_{12},v_{12})=(P,0,Q,-1) \\
& \quad \leftarrow(s_{13},t_{13},u_{13},v_{13}):=(s_{12},t_{12},u_{12},(v_{12}+1)t_{12}^{-1}),
\end{align*}
where $C_{11}$ corresponds to $V_2$, while $C_{12}$ and $C_{13}$ corresponds to $V_4$. We need additional four blowups for $V_2$:
\begin{align*}
C_{14}:&(s_{13},t_{13},u_{13},v_{13})=(0,0,1+h,0) \\
& \quad \leftarrow (s_{14},t_{14},u_{14},v_{14}):=(s_{13}t_{13}^{-1},t_{13},(u_{13}-1-h)t_{13}^{-1},v_{13}t_{13}^{-1}),\\
C_{15}:&(s_{14},t_{14},u_{14},v_{14})=(P,0,-2Q-Ph^{-1} , Q) \\
& \quad \leftarrow (s_{15},t_{15},u_{15},v_{15}):= (s_{14},t_{14}, v_{14} , (u_{14}+2v_{14}+s_{14}h^{-1} )t_{14}^{-1}),\\
C_{16}:&(s_{15},t_{15},u_{15},v_{15})=(P,0, -Ph^{-1} , Q) \\
& \quad \leftarrow (s_{16},t_{16},u_{16},v_{16}):= (s_{15},t_{15},(u_{15}+s_{15}h^{-1})t_{15}^{-1}, v_{15}),\\
C_{17}:&(s_{16},t_{16},u_{16},v_{16})=(P,0,Q, 2^{-1}Q+(1+h)h^{-1} P) \\
&\hspace{-1cm} \quad \leftarrow (s_{17},t_{17},u_{17},v_{17}):= (s_{16},t_{16}, u_{16} , (v_{16}-2^{-1}u_{16}-(1+h)h^{-1}  s_{16})t_{16}^{-1}).
\end{align*}
The (total transform of) exceptional divisor $E_i$ of $i$-th blowup is described in the local chart as
\begin{align*}
E_i:&\ s_i=0, \quad (i=1,2,3,4,5,6,7,8,9,10,11,14)\\
E_i:&\ t_i=0, \quad (i=12,13,15,16,17).      
\end{align*}

Let us denote the total transform (with respect to blowups) of the divisors (hyper-surfaces) $c_0x_0+c_1x_1+a=0$ and $c_2x_2+c_3x_3+b=0$ by $H_a$ and $H_b$ respectively, where $(c_0:c_1:a)$ and  $(c_2:c_3:b)$ are constant  ${\mathbb P}^2$ vectors.  Let us write the classes of $H_a$, $H_b$ and $E_i$ modulo linear equivalence as ${\mathcal H}_a$, ${\mathcal H}_b$ and ${\mathcal E}_i$. Then, the Picard group of this variety ${\mathcal X}$ becomes a ${\mathbb Z}$-module:
\begin{align}
\pic({\mathcal X}) =& \Z \H_a\oplus \Z\H_b \oplus \bigoplus_{i=1}^{17} \Z \E_i .
\end{align}

\begin{theorem}
The map \eqref{4d} is lifted to an algebraically stable map on the rational variety
obtained by blow-ups along $C_i$, $i=1,2,\dots,17$, from $\P^2\times \P^2$.
\end{theorem}
\begin{proof} The algebraic stability can be checked as follows.
In the present case, the indeterminate set $I(\varphi)$ is given by
$$I(\varphi)=\varphi^{-1}(E_6-E_7)\subset E_{11},$$
while the condition that the dimension of $\varphi(D\setminus I(\varphi))$ is at most two
implies $D=E_1-E_2$ and $\varphi(D\setminus I(\varphi))=\varphi(E_1-E_2)\subset E_{11}$.   
It can be checked that $\varphi(E_1-E_2)$ and $I(\varphi^k)$, $k=1,2,3,\dots$,
are different two-dimensional subvarirties in $E_{11}$, and hence \eqref{cond_AS} can not occur.
\end{proof}

The class of proper transform of 
$E_i$ is 
\begin{align*}
&\E_i-\E_{i+1}\quad  (i=1,2,3,4,6,7,8,9,12,13,14,15,16)\\
&\E_i \quad (i=5,10,17), \quad \E_{11}-\E_{15}.
\end{align*}  
Since the defining function of the hyper-surface $z_1=0$ takes zero with multiplicities\\ 
$0, 0, 0, 0,0, 1,2, 2, 2, 2, 1, 1, 1, 2, 2, 2, 2$ on $E_i$ ($i=1,\dots,17$), 
it is decomposed as
\begin{align*}
\H_a=& \mbox{Proper transform} \\
&+ (\E_6-\E_7)+2(\E_7-\E_8)+2(\E_8-\E_9)+2(\E_9-\E_{10})+2\E_{10}\\ 
& +(\E_{11}-\E_{14})+(\E_{12}-\E_{13})+(\E_{13}-\E_{14})+2(\E_{14}-\E_{15})
\\&+2(\E_{15}-\E_{16})+2(\E_{16}-\E_{17})+2\E_{17},
\end{align*}
where each class enclosed in parentheses determines a prime divisor uniquely
(we called such a class deterministic \cite{CDT2017}). 
Hence the class of its proper transform is $\H_a-\E_6-\E_7-\E_{11}-\E_{12}$.
Similarly, the defining function of the hyper-surface $x_2-1=0$ takes zero with multiplicities $1,1,1, 1, 1, 1, 1, 1, 1, 1, 1, 0, 0, 1, 1, 1, 1$ on $E_i$, and therefore the class of its proper transform is $\H_b-\E_1-\E_6-\E_{11}$.
Along the same line, the proper transform of  $z_3=0$ can be computed as $\H_b-\E_1-\E_2-\E_{11}-\E_{12}$.

Using these data, we can compute the pull-back action of Mapping $\varphi$ \eqref{4d} on the Piacard group. For example, the pull-back of $E_1$ is
$(\bar{x}_1,\bar{z}_2,\bar{z}_3)=(0,0,0)$, whose ``common factor'' on each local coordinate system is $x_2-1$, $s_6$, $s_7$, $s_8$ or $s_9$. Thus, we have
\begin{align*}
\varphi(\E_1)=&(\H_2-\E_1-\E_6-\E_{11})+\sum_{i=6}^9 (\E_i-\E_{i+1})\\
=& \H_2-\E_1-\E_{10}-\E_{11}.
\end{align*}       
Along the same line, we have the following proposition.
\begin{proposition}
The pull-back $\varphi^*$ of Mapping \eqref{4d} is a linear action on the Picard group given by
\begin{align*}
&\H_a\to \H_b,\\
&\H_b\to \H_a+3\H_b-2\E_1-3\E_{11}-\E_{6,7,9,10,12,13,14}, \\ 
&\E_1 \to \H_b-\E_{1,10,11}, \quad
\E_2 \to \H_b-\E_{1,9,11}, \quad  \E_3 \to \H_b-\E_{1,7,9,11}+\E_8,\\ 
&\E_4 \to \H_b-\E_{1,7,11}, \quad  \E_5 \to \H_b-\E_{1,6,11},\\
&\E_6 \to \E_{14}, \quad  
\E_7 \to \E_{14}, \quad\E_8 \to \E_{15}, \quad  \E_9 \to \E_{16},
\quad  \E_{10} \to \E_{17},\\
&\E_{11} \to \E_{1,11}-\E_{14}, \quad  \E_{12} \to \H_b-\E_{1,11,13}, 
\quad  \E_{13} \to \H_b-\E_{1,11,12},\\
&\E_{14} \to \E_2, \quad  \E_{15} \to \E_3, \quad  \E_{16} \to \E_4,\quad
\E_{17} \to \E_5,
\end{align*}
where $\E_{i_1,\dots,i_k}$ denotes $\E_{i_1}+\dots+\E_{i_k}$.
The Jordan blocks of the corresponding matrix are
$$1 ,\ -1,\  1^{\frac{1}{3}} \ (3\times 3\mbox{ {\rm blocks}}),\  
\begin{bmatrix}
1&1&0\\0&1&1\\0&0&1
\end{bmatrix},\ 
\begin{bmatrix}
0&1&0&0&0\\
0&0&1&0&0\\
0&0&0&1&0\\
0&0&0&0&1\\
0&0&0&0&0
\end{bmatrix}. 
$$
In particular, the degree of the mapping $\varphi^n$ grows quadratically with respect to $n$.
\end{proposition}

\begin{corollary}
The degree of $\psi^n$ for  the 3-dimensional map $\psi$ \eqref{3d_system} also 
grows quadratically with respect to $n$.
\end{corollary}

\begin{proof} \footnote{This kind of argument is not original. More general results can be found in \cite{Mase2016}, where it is shown that all the reduced systems from classical KP or BKP equation have the quadratic degree growth.}
Let us denote the initial values as $(x_0,x_1,x_2,x_3)=(x_0^{(0)},x_1^{(0)},x_2^{(0)},x_3^{(0)})$. Map $\psi^n$ is obtained by substituting $x_3=h(x_0,x_1,x_2)$ to 
$\varphi^n: x_i^{(n)} = f_i^{(n)}(x_0,x_1,x_2,x_3)$, $i=0,1,2$, where $h$ and $f_i$'s are some rational functions.
Hence the degree of $x_i^{(n)}$'s  with respect to $x_0, x_1, x_2$ are bounded from the above by $(\mbox{degree of $h$}) \times (\mbox{degree of $f_i^{(n)}$})$.
Since the degree of $f_i^{(n)}$'s are quadratic with respect to $n$, the degree of $x_i^{(n)}$'s are at most quadratic. On the other hand, since $\psi$ is a QRT map with respect to $x_0$ and $x_2$, its degree with regarding $x_1$ as a constant grows quadratically \cite{Takenawa2001}, hence the degree of $x_i^{(n)}$'s are at least quadratic.
\end{proof}

The proper transforms of the conserved quantities $I_1$ and $I_2$ are
\begin{align*}
I_1:&\ 2\H_a +2\H_b -2\E_1-2\E_6-4\E_{11}-\E_{2,4,7,9,12,13,14,16}\\
I_2:&\ 2\H_a +2\H_b -3\E_{11}-\E_{1,2,4,5,6,7,9,10,12,13,14,16,17},
\end{align*}
which are preserved by $\varphi^*$.

We can consider the inverse problem.
\begin{proposition}
Hyper-surfaces whose class is $2\H_a +2\H_b -2\E_1-2\E_6-4\E_{11}-\E_{2,4,7,9,12,13,14,16}$ are given by $C_0+C_1I_1=0$ with $(C_0:C_1)\in \P^1$ and $C_1\neq 0$.
Hyper-surfaces whose class is $2\H_a +2\H_b -3\E_{11}-\E_{1,2,4,5,6,7,9,10,12,13,14,16,17}$ are given by  $C_0+C_1I_1+C_2I_2=0$ with $(C_0:C_1:C_2)\in \P^2$ and $C_2\neq 0$.
\end{proposition}

Thus,  we can compute invariants by using the action of the system $\varphi$ on the Picard group. 

\begin{proof}
The proof is straightforward but tedious. For example, the defining polynomials of a curve of the class $2\H_a +2\H_b -2\E_1-2\E_6-4\E_{11}-\E_{2,4,7,9,12,13,14,16}$ can be written as
\begin{align*}
&f(x_0,x_1,x_2,x_3):=\sum_{\begin{array}{c}i_0, i_1, i_2,i_3\geq 0\\ i_0+ i_1+ i_2+i_3 \leq 2\end{array}}
a_{i_0i_1i_2i_3}x_0^{i_0}x_1^{i_1}x_2^{i_2}x_3^{i_3},\\
&z_2^2f(x_0,x_1,z_2z_3^{-1},z_3^{-1})\quad \mbox{around $E_1$},\\
&z_0^2f(z_0z_1^{-1}, z_1^{-1},x_2,x_3)\quad  \mbox{around $E_5$},\\
&z_0^2z_2^2 f(z_0z_1^{-1}, z_1^{-1},z_2z_3^{-1},z_3^{-1})\quad  \mbox{around $E_{11}$}.
\end{align*}
The coefficients are determined so that defining polynomial takes zero with multiplicity
2, 3, 3, 4, 4, 2, 3, 3, 4, 4, 4, 1, 2, 7, 7, 8, 8 on $E_i$'s; which verifies the claim.
\end{proof}

\subsection*{Acknowledgement}
ASC was supported by Program Nuccleu, PN/2019, Romanian Ministery of Education
and T.~T. was supported by the Japan Society for the
Promotion of Science, Grand-in-Aid (C) (17K05271).

\end{document}